\documentclass[11pt]{article}
\usepackage{amsmath}
\usepackage{amsfonts}
\usepackage{amssymb}
\usepackage{multirow}
\usepackage{graphicx}
\usepackage[colorlinks, citecolor=blue]{hyperref}
\usepackage{tikz}
\newtheorem{theorem}{\bf Theorem}[section]

\newtheorem{lemma}[theorem]{\bf Lemma}
\newtheorem{prop}[theorem]{\bf Proposition}

\newtheorem{remark}[theorem]{\bf Remark}
\newenvironment{proof}{\noindent{\em Proof:}}{\quad \hfill$\Box$\vspace{2ex}}
\newenvironment{sequation}{\begin{equation}\small}{\end{equation}}
\newenvironment{seqnarray}{\begin{eqnarray}\small}{\end{eqnarray}}

\makeatletter

\newcommand{\Rmnum}[1]{\expandafter\@slowromancap\romannumeral #1@}
\makeatother
\numberwithin{equation}{section}

\begin{document}
\thispagestyle{empty}
\begin{center}{\LARGE \bf  Totally Real Flat Minimal Surfaces in Quaternionic Projective Spaces}
\end{center}

\begin{center}
Ling He
\footnote{
L. He
\\Center for Applied Mathematics, Tianjin University, Tianjin 300072, P. R. China
\\E-mail: heling@tju.edu.cn
}
and
Xianchao Zhou
\footnote{
X.C. Zhou
\\Department of Applied Mathematics, Zhejiang University of Technology, Hangzhou 310023, P. R. China
\\E-mail: zhouxianch07@zjut.edu.cn
}
\end{center}

\bigskip

\begin{center}
\parbox{12cm}
{\footnotesize{\bf ABSTRACT.}  In this paper, we study totally real minimal surfaces in the quaternionic projective space $\mathbb{H}P^n$.
We prove that the linearly full totally real flat minimal surfaces
of isotropy order $n$ in $\mathbb{H}P^n$ are two surfaces in $\mathbb{C}P^n$, one of which is the Clifford solution,
up to symplectic congruence.}
\end{center}

\bigskip

\section{Introduction}\label{sec1:introduction}

A.Bahy-El-Dien and J.C.Wood have developed
a beautiful and quite complete theory for harmonic two-spheres
in the quaternionic projective space $\mathbb{H}P^n$(cf.\cite{A. Bahy-El-Dien 1991}).
All harmonic two-spheres in $\mathbb{H}P^n$
are generated from quaternionic Frenet pair or quaternionic mixed pair,
which are direct sum of certain harmonic two-spheres in $\mathbb{C}P^{2n+1}$,
by certain flag transforms called forward and backward replacements.
The above theory generates a series of classification results about minimal two-spheres of
constant Gauss curvature in $\mathbb{H}P^n$(cf.\cite{FeiHe},\cite{HeJiao 2014},
\cite{HeJiao 2015},\cite{HeJiao 2015-2}).

It is natural to consider the harmonic
maps from Riemann surfaces of higher genus.
One of such examples is the totally real superconformal minimal tori in $\mathbb{C}P^2$,
H.Ma and Y.J.Ma (cf.\cite{Ma 2005}) described explicitly all these tori by Prym-theta functions.
Generally, there is a family of totally real
flat minimal surfaces in $\mathbb{C}P^n$.
This family was given by K.Kenmotsu (cf.\cite{Kenmotsu 1985})
and J.Bolton and L.M.Woodward (cf.\cite{Bolton-Woodward 1992}) respectively. Later,
G.R.Jensen and R.J.Liao
discovered continuous families of noncongruent flat minimal tori
in $\mathbb{C}P^n$ by this family of minimal surfaces(cf.\cite{Jensen-Liao 1995}).
We will apply this family to characterize totally real flat minimal surfaces in $\mathbb{H}P^n$.

It is well known that Veronese
sequences are a series of harmonic two-spheres
of constant Gauss curvature in $\mathbb{C}P^n$ (cf.\cite{Bando-Ohnita}, \cite{J. Bolton 1988}).
In \cite{HeWang2005}, Y.J.He and C.P.Wang proved that the Veronese
sequences in $\mathbb{R}P^{2m}$ ($2\leq 2m\leq n$) are the only totally real
minimal two-spheres with constant Gauss curvature in $\mathbb{H}P^n$.
In \cite{S.Udagawa 1997}, S.Udagawa studied the classification of harmonic two-tori in $\mathbb{H}P^2$ and $\mathbb{H}P^3$.

In this paper we mainly study totally real minimal surfaces in $\mathbb{H}P^n$.
In section 2, we define quaternionic K\"ahler angle with respect to an
isometric immersion from Riemann surface to $\mathbb{H}P^n$, which
gives a measure of the failure of the immersion
to be a totally complex map or a totally real map.
In section 3, we study totally real minimal surfaces in $\mathbb{H}P^n$ by the method of twistor lift,
and get some properties with respect to harmonic sequence
(see Theorem \ref{thm3-1}, Proposition \ref{prop3-1} and Proposition \ref{prop3-2}).
In section 4, we give a classification theorem about the linearly full totally real flat minimal surfaces of isotropy order $n$ in $\mathbb{H}P^n$ (see Theorem \ref{thm3}).

\section{Surfaces in quaternionic projective spaces}

For any $N=1,2,\ldots,$ let $\left\langle , \right\rangle $ denote
the standard Hermitian inner product on $\mathbb{C}^{N}$ defined by
$\left\langle z, w\right\rangle= z_1\overline w_1 +\ldots+
z_N\overline w_N$, where $z=(z_1, \ldots, z_N)^T, w=(w_1, \ldots,
w_N)^T\in {\mathbb{C}}^N$ and $\bar{}$ denotes complex conjugation.
Let $\mathbb{H}$ denote the division ring of quaternions, i.e.
$$
\mathbb{H}=\left\{a+b\texttt{i}+c\texttt{j}+d\texttt{k}
|a,b,c,d\in\mathbb{R},\texttt{i}^2=\texttt{j}^2=\texttt{k}^2
=\texttt{i}\texttt{j}\texttt{k}=-1
\right\}.
$$
Since $a+b\texttt{i}+c\texttt{j}+d\texttt{k}=(a+b\texttt{i})+(c+d\texttt{i})\texttt{j}$,
then it gives an identification of ${\mathbb{C}}^2$ with $\mathbb{H}$.
Let $n \in\left\{1,2,\cdots \right\} $, we have a corresponding identification
of ${\mathbb{C}}^{2n+2}$ with ${\mathbb{H}}^{n+1}$.
The inner product on ${\mathbb{H}}^{n+1}$ is given by $g_{{\mathbb{H}}^{n+1}}=\text{Re}\left\langle , \right\rangle$.
For any
$z_1+z_2\texttt{j}\in\mathbb{H}$, the left multiplication by $\texttt{j}$ is given by
$\texttt{j}(z_1+z_2\texttt{j})=-\bar{z}_2+\bar{z}_1\texttt{j}$.
Then $\texttt{j}$ induces a conjugate
linear map from ${\mathbb{C}}^{2n+2}$ to ${\mathbb{C}}^{2n+2}$, also denoted by $\texttt{j}$,
i.e.
$$
{\texttt{j}}(z_1, z_2, \ldots,
z_{2n+1}, z_{2n+2})^T= (-\bar{z}_2, \bar{z}_1, \ldots, -\bar{z}_{2n+2},
\bar{z}_{2n+1})^T.
$$
Then ${\texttt{j}}^2=-\mathrm{id}$ where $\mathrm{id}$ denotes
the identity map on ${\mathbb{C}}^{2n+2}$. In fact, for any $v\in
{\mathbb{C}}^{2n+2}$,
$$
{\texttt{j}}v=J_{n+1} \bar{v},
$$
where
$J_{n+1}=diag \underbrace{\left\{ \left( \begin{array} {lr} 0 & -1 \\ 1 & 0 \\
\end{array}\right), \ldots, \left( \begin{array} {lr} 0 & -1 \\ 1 & 0 \\
\end{array}\right)  \right\}}_{n+1}$.
By the above, we immediately have the following lemma (cf. \cite{HeJiao 2014}).
\begin{lemma}\label{lem2-0}
The operator $\texttt{j}$ has the following properties: \\
(i) $\left\langle {\texttt{j}}v,
{\texttt{j}}w\right\rangle=\left\langle w, v\right\rangle$ for all
$v,w\in{\mathbb{C}}^{2n+2}$; \\
(ii) $\left\langle {\texttt{j}}v,
v\right\rangle=0$ for all
$v\in{\mathbb{C}}^{2n+2}$;\\
(iii) $\partial_z \circ {\texttt{j}}=\texttt{j} \circ {\partial}_{\bar{z}}$, \
${\partial}_{\bar{z}} \circ \texttt{j}=\texttt{j} \circ \partial_{z}$;\\
(iv)
$\texttt{j}(\lambda v)=\overline{\lambda} \texttt{j} v$ for any
$\lambda\in{\mathbb{C}}$, $v\in{\mathbb{C}}^{2n+2}$.
\end{lemma}

The quaternionic projective space $\mathbb{H}P^{n}$ is the set
of all one-dimensional quaternionic subspaces of ${\mathbb{H}}^{n+1}$.
Analogous to the Fubini-Study metric on $\mathbb{C}P^{n}$, $\mathbb{H}P^{n}$ also carries a natural Riemannian metric
denoted by $g_{\mathbb{H}P^n}$. Now we give the canonical quaternionic K\"ahler structure
compatible with $g_{\mathbb{H}P^n}$, locally denoted by
$\left\{\texttt{I},\texttt{J},\texttt{K}\right\}$.
Let
\begin{seqnarray}\label{eq2-1}
\pi:S^{4n+3}\left(\subset \mathbb{C}^{2n+2}\right)&\rightarrow & \mathbb{H}P^n\nonumber\\
q&\mapsto & \pi(q)
\end{seqnarray}
be the Riemannian submersion with fiber $S^3$, i.e., for $q=(z_1,z_2,\cdots,z_{2n+1},z_{2n+2})$,
$\pi(q)=[z_1+z_2 \texttt{j},\cdots,z_{2n+1}+z_{2n+2}\texttt{j}]$,
where $|q|^2=\sum_{\alpha=1}^{2n+2}|z_\alpha|^2=1$.
For any $q\in S^{4n+3}$, the tangent space of fiber is
given by
$$
\mathcal{V}_q=\text{span}_{\mathbb{R}}\left\{\texttt{i}q,~\texttt{j}q,\texttt{k}q\right\}.
$$
We define the horizontal space $\mathcal{H}_q$ by
$$
\mathcal{H}_q=\left(\mathcal{V}_q\right)^{\perp}=
\left\{v\in \mathbb{C}^{2n+2}|\left\langle v,q\right\rangle=\left\langle v,\texttt{j}q\right\rangle=0\right\},
$$
which is isomorphic to $\mathbb{H}^{n}$.
Then
$$
T_qS^{4n+3}=\mathcal{H}_q\oplus \mathcal{V}_q
$$
and the tangent map $\pi_*:\mathcal{H}_q\rightarrow T_{\pi(q)}\mathbb{H}P^n$ is an isometric. For any $X\in T_{\pi(q)}\mathbb{H}P^n$, denote the horizontal lift of $X$ by
$X^{\mathcal{H}}=\left(\pi_*\right)^{-1}(X)$,
\begin{sequation}\label{eq2-2}
\texttt{I}(X)=\pi_*(\texttt{i}X^{\mathcal{H}}),
~\texttt{J}(X)=\pi_*(\texttt{j}X^{\mathcal{H}}),
~\texttt{K}(X)=\pi_*(\texttt{k}X^{\mathcal{H}}).
\end{sequation}
Let $\left\{\omega_{\texttt{I}},\omega_{\texttt{J}},\omega_{\texttt{K}}\right\}$
be the two-forms related to the quaternionic K$\ddot{\text{a}}$hler structure
$\left\{\texttt{I},\texttt{J},\texttt{K}\right\}$, defined as follows
\begin{sequation}\label{eq2-3}
\omega_{\texttt{I}}(X,Y)=g_{\mathbb{H}P^n}(\texttt{I}X,Y),
~\omega_{\texttt{J}}(X,Y)=g_{\mathbb{H}P^n}(\texttt{J}X,Y),
~\omega_{\texttt{K}}(X,Y)=g_{\mathbb{H}P^n}(\texttt{K}X,Y),
\end{sequation}
where $X,Y\in T_{\pi(q)}\mathbb{H}P^n$.
We know every two-form of $\left\{\omega_{\texttt{I}},\omega_{\texttt{J}},\omega_{\texttt{K}}\right\}$
is defined locally, but
\begin{sequation}\label{eq2-4}
\Omega=\omega_{\texttt{I}}\wedge\omega_{\texttt{I}}
+\omega_{\texttt{J}}\wedge\omega_{\texttt{J}}
+\omega_{\texttt{K}}\wedge\omega_{\texttt{K}}
\end{sequation}
is a globally defined, non-degenerate four-form on $\mathbb{H}P^n$(cf.\cite{Salamon 1982},\cite{Besse 1987}, \cite{LiZhang 2005}).
It is usually called the fundamental four-form of $\mathbb{H}P^n$.

Let $\varphi:M\rightarrow \mathbb{H}P^n$ be an isometric immersion from Riemann surface $M$ to $\mathbb{H}P^n$.
Then on some open set $U$ of $M$, $\varphi$ has a natural local lift
$s:U\rightarrow S^{4n+3}$.  The explicit description is that the following diagram commutes:
\begin{equation*}
\begin{tikzpicture}[->,>=stealth,auto,node distance=5em]
\centering
\node (X) {$U\subset M$};
\node (Y) [right of=X] {$\mathbb{H}P^n$};
\node (T) [above of=Y] {$S^{4n+3}$};
\draw (X) -- (Y) node[midway, below] {$\varphi$};
\draw (X) -- (T) node[midway, above] {$s$};
\draw (T) -- (Y) node[midway, right] {$\pi$};
\end{tikzpicture}
\end{equation*}
That is, if
$\varphi=[z_1+z_2 \texttt{j},\cdots,z_{2n+1}+z_{2n+2}\texttt{j}]^T$
with $\sum_{\alpha=1}^{2n+2}|z_{\alpha}|^2=1$, then
\\$s=(z_1,z_2,\cdots,z_{2n+1},z_{2n+2})^T$.
Obviously it satisfies $\varphi=\pi\circ s$.

Let $z=x+\texttt{i}y$ be the local coordinate of $M$ such that
the metric induced by $\varphi$ is given by
\begin{sequation}\label{eq2-5}
g_M=\varphi^*g_{\mathbb{H}P^n}=e^{2u}dzd\bar{z}=e^{2u}(dx^2+dy^2),
\end{sequation}
where $u$ is a real function.
Set $\theta_1=e^udx,~\theta_2=e^udy$, $e_1=e^{-u}\frac{\partial}{\partial x}$,
$e_2=e^{-u}\frac{\partial}{\partial y}$. Then
$\left\{e_1,e_2\right\}$ is a standard orthogonal basis of $TM$,
and $\left\{\theta_1,\theta_2\right\}$ is its dual basis.
For given isometric immersion $\varphi$ and the corresponding
lift $s$, we define a new differential operator $d^{\mathcal{H}}$ called
{\it horizontal differential operator} by
$$
d^{\mathcal{H}}s
=ds
-\left\langle ds,s\right\rangle s
-\left\langle ds,\texttt{j}s\right\rangle \texttt{j}s,
$$
where $d$ is the usual differential operator.
The lift $s$ is called a {\it horizontal lift} if $s$
satisfies $ds=d^{\mathcal{H}}s$, that is,
$\left\langle ds,s\right\rangle
=\left\langle ds,\texttt{j}s\right\rangle=0$.
In the following, all the tangent maps such as $ds,~\pi_*$
are extended by complex linearity.
A straightforward calculation shows
\begin{seqnarray}\label{eq2-6}
\varphi_*e_1
&=&\pi_* s_{*}e_1
=\pi_*\left(s_{*}e_1\right)^{\mathcal{H}}\nonumber\\
&=&e^{-u}\pi_*\left[d^{\mathcal{H}}s\left(\frac{\partial}{\partial z}\right)
+d^{\mathcal{H}}s\left(\frac{\partial}{\partial \bar{z}}\right)\right].
\end{seqnarray}
Similarly, we have
\begin{seqnarray}\label{eq2-7}
\varphi_*e_2
=e^{-u}\pi_*\left[\texttt{i}\left(d^{\mathcal{H}}s\left(\frac{\partial}{\partial z}\right)
-d^{\mathcal{H}}s\left(\frac{\partial}{\partial \bar{z}}\right)\right)\right].
\end{seqnarray}
Since $\varphi$ and $\pi_*\mid_{\mathcal{H}_q}$ are isometric,
then we have by \eqref{eq2-6}
\begin{small}
\begin{eqnarray}\label{eq2-8}
g_M(e_1,e_1)
&=&g_{\mathbb{H}P^n}\left(\varphi_*e_1,\varphi_*e_1\right)\nonumber\\
&=&g_{\mathbb{H}^{n+1}}\left((\pi_*)^{-1}\varphi_*e_1,(\pi_*)^{-1}\varphi_*e_1\right)\nonumber\\
&=&e^{-2u}\textrm{Re}\left\langle d^{\mathcal{H}}s\left(\frac{\partial}{\partial z}\right)
+d^{\mathcal{H}}s\left(\frac{\partial}{\partial \bar{z}}\right),
d^{\mathcal{H}}s\left(\frac{\partial}{\partial z}\right)
+d^{\mathcal{H}}s\left(\frac{\partial}{\partial \bar{z}}\right)\right\rangle\nonumber\\
&=&e^{-2u}\left\{\left|d^{\mathcal{H}}s\left(\frac{\partial}{\partial z}\right)\right|^2
+\left|d^{\mathcal{H}}s\left(\frac{\partial}{\partial \bar{z}}\right)\right|^2
+\left\langle d^{\mathcal{H}}s\left(\frac{\partial}{\partial z}\right),
d^{\mathcal{H}}s\left(\frac{\partial}{\partial \bar{z}}\right)\right\rangle\right.\nonumber\\
&&\left.+\left\langle d^{\mathcal{H}}s\left(\frac{\partial}{\partial \bar{z}}\right),
d^{\mathcal{H}}s\left(\frac{\partial}{\partial z}\right)\right\rangle
\right\}.
\end{eqnarray}
\end{small}
Similarly, we have by \eqref{eq2-6} and \eqref{eq2-7}
\begin{small}
\begin{eqnarray}\label{eq2-9}
g_M(e_2,e_2)
&=&e^{-2u}\left\{\left|d^{\mathcal{H}}s\left(\frac{\partial}{\partial z}\right)\right|^2
+\left|d^{\mathcal{H}}s\left(\frac{\partial}{\partial \bar{z}}\right)\right|^2
-\left\langle d^{\mathcal{H}}s\left(\frac{\partial}{\partial z}\right),
d^{\mathcal{H}}s\left(\frac{\partial}{\partial \bar{z}}\right)\right\rangle\right.\nonumber\\
&&\left.-\left\langle d^{\mathcal{H}}s\left(\frac{\partial}{\partial \bar{z}}\right),
d^{\mathcal{H}}s\left(\frac{\partial}{\partial z}\right)\right\rangle
\right\},
\end{eqnarray}
\end{small}
and
\begin{small}
\begin{eqnarray}\label{eq2-10}
g_M(e_1,e_2)
=e^{-2u}\texttt{i}\left\{\left\langle d^{\mathcal{H}}s\left(\frac{\partial}{\partial z}\right),
d^{\mathcal{H}}s\left(\frac{\partial}{\partial \bar{z}}\right)\right\rangle
-\left\langle d^{\mathcal{H}}s\left(\frac{\partial}{\partial \bar{z}}\right),
d^{\mathcal{H}}s\left(\frac{\partial}{\partial z}\right)\right\rangle
\right\}.
\end{eqnarray}
\end{small}
From $g_M(e_1,e_1)=g_{M}(e_2,e_2)=1$ and $g_{M}(e_1,e_2)=0$, we get by \eqref{eq2-8}-\eqref{eq2-10}
\begin{small}
\begin{eqnarray}\label{eq2-11}
e^{2u}
=\left|d^{\mathcal{H}}s\left(\frac{\partial}{\partial z}\right)\right|^2
+\left|d^{\mathcal{H}}s\left(\frac{\partial}{\partial \bar{z}}\right)\right|^2,
~\left\langle d^{\mathcal{H}}s\left(\frac{\partial}{\partial z}\right),
d^{\mathcal{H}}s\left(\frac{\partial}{\partial \bar{z}}\right)\right\rangle
=0.
\end{eqnarray}
\end{small}
In term of the three two-forms
$\left\{\omega_{\texttt{I}},\omega_{\texttt{J}},\omega_{\texttt{K}}\right\}$,
we can define three angles ${\alpha_1},{\alpha_2},{\alpha_3}$ as follows,
\begin{seqnarray}\label{eq2-13}
\cos{\alpha_1}=\frac{\varphi^*(\omega_{\texttt{I}})}{d\mu_M},
\cos{\alpha_2}=\frac{\varphi^*(\omega_{\texttt{J}})}{d\mu_M},
\cos{\alpha_3}=\frac{\varphi^*(\omega_{\texttt{K}})}{d\mu_M},
\end{seqnarray}
where $d\mu_M=\theta_1\wedge\theta_2$ is the volume form of the metric $g_M$.
Here ${\alpha_1},{\alpha_2},{\alpha_3}$ are locally defined.
But since $
\Omega=\omega_{\texttt{I}}\wedge\omega_{\texttt{I}}
+\omega_{\texttt{J}}\wedge\omega_{\texttt{J}}
+\omega_{\texttt{K}}\wedge\omega_{\texttt{K}}
$
is a globally defined, non-degenerate four-form on $\mathbb{H}P^n$,
then we can define a global function ${\alpha}:M\rightarrow [0,\pi]$ by
\begin{seqnarray}\label{eq2-14}
\cos^2{\alpha}=\frac{\left(\varphi^*(\omega_{\texttt{I}})\right)^2+
\left(\varphi^*(\omega_{\texttt{J}})\right)^2+
\left(\varphi^*(\omega_{\texttt{K}})\right)^2}{\left(d\mu_M\right)^2}.
\end{seqnarray}
We call $\alpha$ the {\it{quaternionic K\"{a}hler angle}}
with respect to the isometric immersion $\varphi$.
Obviously $\cos^2{\alpha}=\cos^2{\alpha_1}+\cos^2{\alpha_2}+\cos^2{\alpha_3}$.
In the following we compute the explicit expression of $\cos^2{\alpha}$.
A straightforward calculation shows
\begin{seqnarray}\label{eq2-15}
\cos{\alpha_1}
&=&\left(\varphi^*\omega_{\texttt{I}}\right)(e_1,e_2)
=\omega_{\texttt{I}}\left(\varphi_*e_1,\varphi_*e_2\right)\nonumber\\
&=&g_{\mathbb{H}P^n}\left(\texttt{I}\varphi_*e_1,\varphi_*e_2\right)
=g_{\mathbb{H}^{n+1}}\left(\texttt{i}\left(\varphi_*e_1\right)^{\mathcal{H}},\left(\varphi_*e_2\right)^{\mathcal{H}}\right)\nonumber\\
&=&e^{-2u}\textrm{Re}\left\langle \texttt{i}\left[d^{\mathcal{H}}s\left(\frac{\partial}{\partial z}\right)
+d^{\mathcal{H}}s\left(\frac{\partial}{\partial \bar{z}}\right)\right],
\texttt{i}\left[d^{\mathcal{H}}s\left(\frac{\partial}{\partial z}\right)
-d^{\mathcal{H}}s\left(\frac{\partial}{\partial \bar{z}}\right)\right]\right\rangle\nonumber\\
&=&\frac{\left|d^{\mathcal{H}}s\left(\frac{\partial}{\partial z}\right)\right|^2
-\left|d^{\mathcal{H}}s\left(\frac{\partial}{\partial \bar{z}}\right)\right|^2}
{\left|d^{\mathcal{H}}s\left(\frac{\partial}{\partial z}\right)\right|^2
+\left|d^{\mathcal{H}}s\left(\frac{\partial}{\partial \bar{z}}\right)\right|^2},
\end{seqnarray}
\begin{seqnarray}\label{eq2-16}
\cos{\alpha_2}
&=&\left(\varphi^*\omega_{\texttt{J}}\right)(e_1,e_2)
=\omega_{\texttt{J}}\left(\varphi_*e_1,\varphi_*e_2\right)\nonumber\\
&=&g_{\mathbb{H}P^n}\left(\texttt{J}\varphi_*e_1,\varphi_*e_2\right)
=g_{\mathbb{H}^{n+1}}\left(\texttt{j}\left(\varphi_*e_1\right)^{\mathcal{H}},\left(\varphi_*e_2\right)^{\mathcal{H}}\right)\nonumber\\
&=&e^{-2u}\textrm{Re}\left\langle \texttt{j}\left[d^{\mathcal{H}}s\left(\frac{\partial}{\partial z}\right)
+d^{\mathcal{H}}s\left(\frac{\partial}{\partial \bar{z}}\right)\right],
\texttt{i}\left[d^{\mathcal{H}}s\left(\frac{\partial}{\partial z}\right)
-d^{\mathcal{H}}s\left(\frac{\partial}{\partial \bar{z}}\right)\right]\right\rangle\nonumber\\
&=&e^{-2u}2\textrm{Re}\left[-\texttt{i}\left\langle
\texttt{j}d^{\mathcal{H}}s\left(\frac{\partial}{\partial \bar{z}}\right),
d^{\mathcal{H}}s\left(\frac{\partial}{\partial z}\right)\right\rangle\right]\nonumber\\
&=&\frac{\texttt{i}\left[\left\langle d^{\mathcal{H}}s\left(\frac{\partial}{\partial {z}}\right),
\texttt{j}d^{\mathcal{H}}s\left(\frac{\partial}{\partial \bar{z}}\right)\right\rangle
-\left\langle \texttt{j}d^{\mathcal{H}}s\left(\frac{\partial}{\partial \bar{z}}\right),
d^{\mathcal{H}}s\left(\frac{\partial}{\partial {z}}\right)\right\rangle\right]}
{\left|d^{\mathcal{H}}s\left(\frac{\partial}{\partial z}\right)\right|^2
+\left|d^{\mathcal{H}}s\left(\frac{\partial}{\partial \bar{z}}\right)\right|^2},
\end{seqnarray}
and
\begin{seqnarray}\label{eq2-17}
\cos{\alpha_3}
&=&\left(\varphi^*\omega_{\texttt{K}}\right)(e_1,e_2)
=\omega_{\texttt{K}}\left(\varphi_*e_1,\varphi_*e_2\right)\nonumber\\
&=&g_{\mathbb{H}P^n}\left(\texttt{K}\varphi_*e_1,\varphi_*e_2\right)
=g_{\mathbb{H}^{n+1}}\left(\texttt{k}\left(\varphi_*e_1\right)^{\mathcal{H}},\left(\varphi_*e_2\right)^{\mathcal{H}}\right)\nonumber\\
&=&e^{-2u}\textrm{Re}\left\langle \texttt{ij}\left[d^{\mathcal{H}}s\left(\frac{\partial}{\partial z}\right)
+d^{\mathcal{H}}s\left(\frac{\partial}{\partial \bar{z}}\right)\right],
\texttt{i}\left[d^{\mathcal{H}}s\left(\frac{\partial}{\partial z}\right)
-d^{\mathcal{H}}s\left(\frac{\partial}{\partial \bar{z}}\right)\right]\right\rangle\nonumber\\
&=&e^{-2u}2\textrm{Re}\left[\left\langle
\texttt{j}d^{\mathcal{H}}s\left(\frac{\partial}{\partial \bar{z}}\right),
d^{\mathcal{H}}s\left(\frac{\partial}{\partial z}\right)\right\rangle\right]\nonumber\\
&=&\frac{\left\langle d^{\mathcal{H}}s\left(\frac{\partial}{\partial {z}}\right),
\texttt{j}d^{\mathcal{H}}s\left(\frac{\partial}{\partial \bar{z}}\right)\right\rangle
+\left\langle \texttt{j}d^{\mathcal{H}}s\left(\frac{\partial}{\partial \bar{z}}\right),
d^{\mathcal{H}}s\left(\frac{\partial}{\partial {z}}\right)\right\rangle}
{\left|d^{\mathcal{H}}s\left(\frac{\partial}{\partial z}\right)\right|^2
+\left|d^{\mathcal{H}}s\left(\frac{\partial}{\partial \bar{z}}\right)\right|^2}.
\end{seqnarray}
Then we get by \eqref{eq2-15}-\eqref{eq2-17}
\begin{seqnarray}\label{eq2-18}
\cos^2{\alpha}
=\frac{\left(\left|d^{\mathcal{H}}s\left(\frac{\partial}{\partial z}\right)\right|^2
-\left|d^{\mathcal{H}}s\left(\frac{\partial}{\partial \bar{z}}\right)\right|^2\right)^2
+4\left|\left\langle d^{\mathcal{H}}s\left(\frac{\partial}{\partial {z}}\right),
\texttt{j}d^{\mathcal{H}}s\left(\frac{\partial}{\partial \bar{z}}\right)\right\rangle\right|^2}
{\left(\left|d^{\mathcal{H}}s\left(\frac{\partial}{\partial z}\right)\right|^2
+\left|d^{\mathcal{H}}s\left(\frac{\partial}{\partial \bar{z}}\right)\right|^2\right)^2}.
\end{seqnarray}
From \eqref{eq2-18}, we find
\begin{small}
$$
\cos^2{\alpha}=1-4
\frac{\left|d^{\mathcal{H}}s\left(\frac{\partial}{\partial z}\right)\right|^2
\left|d^{\mathcal{H}}s\left(\frac{\partial}{\partial \bar{z}}\right)\right|^2
-\left|\left\langle d^{\mathcal{H}}s\left(\frac{\partial}{\partial {z}}\right),
\texttt{j}d^{\mathcal{H}}s\left(\frac{\partial}{\partial \bar{z}}\right)\right\rangle\right|^2}
{\left(\left|d^{\mathcal{H}}s\left(\frac{\partial}{\partial z}\right)\right|^2
+\left|d^{\mathcal{H}}s\left(\frac{\partial}{\partial \bar{z}}\right)\right|^2\right)^2}
\leq 1.
$$
\end{small}

We know there are two kinds of typical immersed surfaces in $\mathbb{H}P^n$: totally complex
 and totally real(cf. \cite{S.Funabashi 1978}\cite{S.Funabashi 1979}\cite{K.Tsukada 1985}).
Let $\varphi:M\rightarrow \mathbb{H}P^n$ be an isometric immersion of a Riemann surface $M$
into $\mathbb{H}P^n$.
Then {\it{quaternionic K\"{a}hler angle}} gives a measure of the failure of $\varphi$
to be a totally complex map or a totally real map.
Indeed $\varphi$ is totally complex if and only if $\alpha(p)=0$ for all $p\in M$,
while $\varphi$ is totally real if and only if $\alpha(p)=\frac{\pi}{2}$ for all $p\in M$.
From the above discussions, we get the following conclusion,
which firstly appeared in (\cite{HeWang2005}, Lemma 1).
\begin{lemma}
Let $\varphi:M\rightarrow \mathbb{H}P^n$ be a totally real isometric immersion.
Let $s$ be a local lift of $\varphi$ satisfying $|s|=1$.
Then
\begin{sequation}\label{eq+0}
\left|d^{\mathcal{H}}s\left(\frac{\partial}{\partial z}\right)\right|^2
=\left|d^{\mathcal{H}}s\left(\frac{\partial}{\partial \bar{z}}\right)\right|^2,
~\left\langle d^{\mathcal{H}}s\left(\frac{\partial}{\partial {z}}\right),
\texttt{j}d^{\mathcal{H}}s\left(\frac{\partial}{\partial \bar{z}}\right)\right\rangle=0.
\end{sequation}
\label{lem2-1}
\end{lemma}

\section{Totally real minimal surfaces}
We suppose that $(M,~ds^2_M)$ is a simply connected domain in the complex plane $\mathbb{C}$
with local coordinate $\{z\}$. Denote
$$\partial_z=\frac{\partial}{\partial z},
\quad {\partial}_{\bar{z}}=\frac{\partial} {\partial\overline{z}}.$$
The following lemmas will be used in this section.
\begin{lemma}\label{lem3-01}
A map $[s]:M\rightarrow \mathbb{C}P^{2n+1}$ is an isometric minimal immersion
satisfying $|s|^2=1$
if and only if
\begin{sequation}\label{lem3-01-1}
\partial_{\bar{z}}\partial_z s
-\left\langle\partial_{\bar{z}}\partial_z s,s\right\rangle s
-\left\langle\partial_{z}s,s\right\rangle \partial_{\bar{z}}s
-\left\langle\partial_{\bar{z}}s,s\right\rangle \partial_{z}s
-2\left|\left\langle\partial_{z}s,s\right\rangle\right|^2s
=0
\end{sequation}
holds.
\end{lemma}
\begin{proof}
Let $s$ be a column vector valued in $\mathbb{C}P^{2n+1}$ satisfying $|s|^2=1$. Set $f=ss^*$, where $*$ denotes conjugate transpose.
We consider $\mathbb{C}P^{2n+1}$ as a totally geodesic submanifold
of $U(2n+2)$ by Cartan imbedding $\tau(f)=2f-I\in U(2n+2)$ for all
$f\in \mathbb{C}P^{2n+1}$.
Here the bi-invariant metric on $U(n+1)$ is given by
$ds_{U(n+1)}^2=\frac{1}{8}\text{tr}\omega\omega^*$,
where $\omega=g^{-1}dg$ is the Maurer-Cartan form of $U(n+1)$.
The metric on $\mathbb{C}P^{2n+1}$ induced by $\tau$ is the Fubini-Study
metric of constant holomorphic sectional curvature $4$.

Since $\tau$ is totally geodesic, $[s]$ is an isometric minimal immersion
if and only if $\tau\circ [s]$ is an isometric minimal immersion, if and only if the following equation
\begin{sequation}\label{lem3-01-2}
{\partial}_{\bar{z}}A_z=\left[A_z,A_{\bar{z}}\right]
\end{sequation}
holds (cf.\cite{K. Uhlenbeck 1989}), where $A_z=(2f-I)\partial_z f$,
$A_{\bar{z}}=(2f-I)\partial_{\bar{z}} f$, and
$\left[,\right]$ is Lie bracket of $\text{gl}(2n+2,\mathbb{C})$.

Let $s^{\perp}(\partial_z s)=\partial_z s-\left\langle\partial_z s,~s\right\rangle s$, $s^{\perp}(\partial_{\bar{z}} s)=\partial_{\bar{z}} s-\left\langle\partial_{\bar{z}} s,~s\right\rangle s$.
Then,
\begin{sequation}\label{lem3-01-3}
A_z=s\left[s^{\perp}(\partial_{\bar{z}} s)\right]^*-
\left[s^{\perp}(\partial_{z} s)\right]s^*,
~A_{\bar{z}}=s\left[s^{\perp}(\partial_{z} s)\right]^*-
\left[s^{\perp}(\partial_{\bar{z}} s)\right]s^*.
\end{sequation}
Substituting \eqref{lem3-01-3} into \eqref{lem3-01-2},
we obtain
\begin{seqnarray}\label{lem3-01-5}
&&(\partial_{\bar{z}} s)\left[s^{\perp}(\partial_{\bar{z}} s)\right]^*
+s\left[\partial_{z}s^{\perp}(\partial_{\bar{z}} s)\right]^*
-\partial_{\bar{z}}\left[s^{\perp}(\partial_{z} s)\right]s^*
-\left[s^{\perp}(\partial_{z} s)\right]\left(\partial_{z} s\right)^*
+\left|s^{\perp}(\partial_{\bar{z}} s)\right|^2ss^*
\nonumber\\
&&+\left[s^{\perp}(\partial_{z} s)\right]\left[s^{\perp}(\partial_{z} s)\right]^*
-\left|s^{\perp}(\partial_{z} s)\right|^2ss^*
-\left[s^{\perp}(\partial_{\bar{z}} s)\right]\left[s^{\perp}(\partial_{\bar{z}} s)\right]^*=0,
\end{seqnarray}
which is equivalent to
\begin{seqnarray}\label{lem3-01-6}
&&
\partial_{\bar{z}}\left[s^{\perp}(\partial_{z} s)\right]
-\left\langle\partial_{\bar{z}} s,~s\right\rangle\left[s^{\perp}(\partial_{z} s)\right]
+\left|s^{\perp}(\partial_{z} s)\right|^2s=0.
\end{seqnarray}
Thus \eqref{lem3-01-1} follows from \eqref{lem3-01-6}.
\end{proof}

\begin{lemma}\label{lem3-02}
A map $[s]:M\rightarrow \mathbb{C}P^{2n+1}$ is a totally real isometric immersion
satisfying $|s|^2=1$
if and only if
\begin{sequation}\label{lem3-02-1}
\left|\partial_z s\right|^2
=\left|\partial_{\bar{z}}s\right|^2
\end{sequation}
holds.
\end{lemma}
\begin{proof}
The K\"ahler angle of $[s]$ is a function $\theta:M\rightarrow [0,\pi]$ given
in terms of the complex coordinate $z$ on $M$ by (cf. \cite{J. Bolton 1988})
$
\tan\frac{\theta}{2}=\frac{\left|s^{\perp}(\partial_{\bar{z}} s)\right|}{\left|s^{\perp}(\partial_{z} s)\right|}.
$
Since $[s]$ is totally real, then $\theta=\frac{\pi}{2}$, which implies
$
\left|s^{\perp}(\partial_{\bar{z}} s)\right|^2=\left|s^{\perp}(\partial_{z} s)\right|^2.
$
It is equivalent to \eqref{lem3-02-1} by $|s|^2=1$.
\end{proof}

Similar to Lemma \ref{lem3-01}, we have the following conclusion, which firstly appeared in (\cite{HeWang2005}, Proposition 1).
\begin{lemma}\label{lem3-03}
Let $\varphi:M\rightarrow \mathbb{H}P^n$ be an isometric immersion.
Let $s$ be a local lift of $\varphi$ satisfying $|s|=1$,
then $\varphi$ is minimal
if and only if
\begin{small}
\begin{equation}\label{lem3-03-1}
\partial_{\bar{z}}\left[d^{\mathcal{H}}s({\partial}_z)\right]
-\left\langle\partial_{\bar{z}}s,s\right\rangle d^{\mathcal{H}}s({\partial}_z)
-\left\langle\partial_{\bar{z}}s,\texttt{j}s\right\rangle \texttt{j}d^{\mathcal{H}}s({\partial}_{\bar{z}})
+|d^{\mathcal{H}}s({\partial}_z)|^2s
+\left\langle d^{\mathcal{H}}s({\partial}_z),
\texttt{j}d^{\mathcal{H}}s({\partial}_{\bar{z}})\right\rangle\texttt{j}s
=0
\end{equation}
\end{small}
holds.
\end{lemma}
\begin{proof}
Let $\varphi:M\rightarrow \mathbb{H}P^n$ be an isometric immersion.
Let $s$ be a column vector valued in $\mathbb{C}^{2n+2}$ satisfying $|s|=1$,
which is a natural local lift of $\varphi$. Set $\varphi=ss^*+(\texttt{j}s)(\texttt{j}s)^*$.
We consider $\mathbb{H}P^{n}$ as a totally geodesic submanifold
of $U(2n+2)$ by Cartan imbedding $\tau(\varphi)=2\varphi-I\in U(2n+2)$ for all
$\varphi\in \mathbb{H}P^{n}$.
Then $\varphi$ is minimal
if and only if $\tau\circ\varphi$ is minimal. It follows that \eqref{lem3-01-2} holds,
here
\begin{small}
$$
A_z=-\left[d^{\mathcal{H}}s({\partial}_z)\right]s^*+s\left[d^{\mathcal{H}}s({\partial}_{\bar{z}})\right]^*
-\left[\texttt{j}d^{\mathcal{H}}s({\partial}_{\bar{z}})\right](\texttt{j}s)^*
+(\texttt{j}s)\left[\texttt{j}d^{\mathcal{H}}s({\partial}_{z})\right]^*,
$$
\end{small}
and
\begin{small}
$$
A_{\bar{z}}=-\left[d^{\mathcal{H}}s({\partial}_{\bar{z}})\right]s^*+s\left[d^{\mathcal{H}}s({\partial}_{z})\right]^*
-\left[\texttt{j}d^{\mathcal{H}}s({\partial}_{z})\right](\texttt{j}s)^*
+(\texttt{j}s)\left[\texttt{j}d^{\mathcal{H}}s({\partial}_{\bar{z}})\right]^*.
$$
\end{small}
Thus \eqref{lem3-03-1} follows from \eqref{lem3-01-2}.
\end{proof}

The complex projective space $\mathbb{C}P^{2n+1}$ is the twistor
space of $\mathbb{H}P^n$. The twistor map $t:\mathbb{C}P^{2n+1}
\rightarrow \mathbb{H}P^n$, is given by
$
t\left([z_1,z_2,\cdots,z_{2n+1},z_{2n+2}]\right)
=[z_1+z_2 \texttt{j},\cdots,z_{2n+1}+z_{2n+2}\texttt{j}].
$
Then $t$ is a Riemann submersion and the horizontal distribution
is given by
$
\sum_{i=1}^{n+1}\left(z_{2i-1}dz_{2i}-z_{2i}dz_{2i-1}\right)=0.
$

A part of the following result appeared in \cite{HeWang2005}.
For completeness we will give the whole proof.
\begin{theorem}
Let $\varphi:M\rightarrow \mathbb{H}P^n$
be a linearly full totally real isometric minimal immersion, then there exists a totally real isometric horizontal minimal lift
$[s]:M\rightarrow \mathbb{C}P^{2n+1}$.
\label{thm3-1}
\end{theorem}
\begin{proof}
Let $\varphi:M\rightarrow \mathbb{H}P^n$
be a linearly full totally real isometric minimal immersion.
Let $[s]:U\subset M\rightarrow \mathbb{C}P^{2n+1}$ be a natural local lift of $\varphi$ satisfying $|s|=1$.
Here $s$ is a column vector valued in $\mathbb{C}^{2n+2}$.
Set
\begin{sequation}\label{sec4-2-1}
D=\left(s,\texttt{j}s\right)^*\partial_z\left(s,\texttt{j}s\right),
\end{sequation}
then $D$ is a $2\times 2$-matrix.
Then we have
\begin{sequation}
\begin{cases}
\partial_z\left(s,\texttt{j}s\right)=\left(s,\texttt{j}s\right)D+
\left(d^{\mathcal{H}}s({\partial}_z),\texttt{j}d^{\mathcal{H}}s({\partial}_{\bar{z}})\right),\\
\partial_{\bar{z}}\left(s,\texttt{j}s\right)=-\left(s,\texttt{j}s\right)D^*
+\left(d^{\mathcal{H}}s({\partial}_{\bar{z}}),\texttt{j}d^{\mathcal{H}}s({\partial}_z)\right),
\end{cases}
\label{sec4-2}
\end{sequation}
where $\left(d^{\mathcal{H}}s({\partial}_z),\texttt{j}d^{\mathcal{H}}s({\partial}_{\bar{z}})\right)^*
\left(s,\texttt{j}s\right)=\left(d^{\mathcal{H}}s({\partial}_{\bar{z}}),\texttt{j}d^{\mathcal{H}}s({\partial}_z)\right)^*
\left(s,\texttt{j}s\right)=0$.\\
From \eqref{sec4-2} and the identity $\partial_z\partial_{\bar{z}}=\partial_{\bar{z}}\partial_z$, we get
\begin{small}
$$
\partial_{\bar{z}}D+\partial_z D^*+[D,D^*]=
\begin{pmatrix}|d^{\mathcal{H}}s({\partial}_z)|^2-|d^{\mathcal{H}}s({\partial}_{\bar{z}})|^2,
\quad 2\left\langle \texttt{j}d^{\mathcal{H}}s({\partial}_{\bar{z}}), d^{\mathcal{H}}s({\partial}_z)\right\rangle\\
2\left\langle d^{\mathcal{H}}s({\partial}_z),\texttt{j}d^{\mathcal{H}}s({\partial}_{\bar{z}})\right\rangle,
\quad |d^{\mathcal{H}}s({\partial}_{\bar{z}})|^2-|d^{\mathcal{H}}s({\partial}_z)|^2\end{pmatrix}.
$$
\end{small}

Since $\varphi$ is totally real, then by Lemma \ref{lem2-1} we have
\begin{sequation}
\partial_{\bar{z}}D+\partial_z D^*+[D,D^*]=0.
\label{sec4-4}
\end{sequation}
Let $\tilde{s}$ be another local lift of $\varphi$ satisfying $|\tilde{s}|=1$,
then
\begin{sequation}\label{sec4-4+1}
\left(\tilde{s},\texttt{j}\tilde{s}\right)=\left(s,\texttt{j}{s}\right)T,
\end{sequation}
where $T:U\rightarrow SU(2)$ is to be determined such that $\tilde{s}$ is horizontal, i.e. $\left(\tilde{s},\texttt{j}\tilde{s}\right)^*d\left(\tilde{s},\texttt{j}\tilde{s}\right)=0$.
Such $T$ is a solution of the linear PDE
\begin{sequation}
dT+(Ddz-D^*d\overline{z})T=0.
\label{sec4-5}
\end{sequation}
The integrable condition of \eqref{sec4-5} is just \eqref{sec4-4},
so it has a unique solution locally on $M$ for any given initial value.
Let $T$ be a solution of \eqref{sec4-5} with the initial value $T(0)\in SU(2)$.
From \eqref{sec4-5} we have $d(T^*T)=0$ and $d|T|=0$, so $T\in SU(2)$.
Hence there exists a local horizontal lift of $\varphi$, also denoted by $s$.

Let $s$ and $\tilde{s}$ be two horizontal lifts of $\varphi$ on $U$, then
by \eqref{sec4-2-1} we have $D=0$. We define $T$ by \eqref{sec4-4+1}.
It follows from \eqref{sec4-5} that $T:U\rightarrow SU(2)$ is a constant map.
Since $M$ is simply connected, we can extend $s$ to a global horizontal lift
of $\varphi$ on $M$.

Since $\varphi$ is minimal ,
then from Lemma \ref{lem3-03} we know \eqref{lem3-03-1} holds.
Since $s$ is horizontal, then we have
\begin{sequation}\label{eq+1+1}
d^{\mathcal{H}}s=ds,
\left\langle\partial_{z}s,s\right\rangle
=\left\langle\partial_{z}s,\texttt{j}s\right\rangle
=\left\langle\partial_{\bar{z}}s,s\right\rangle
=\left\langle\partial_{\bar{z}}s,\texttt{j}s\right\rangle=0.
\end{sequation}
Substituting \eqref{eq+1+1} into \eqref{lem3-03-1} and using \eqref{eq+0}, we have
\begin{sequation}\label{eq+2}
\partial_{\bar{z}}{\partial}_z s
+|{\partial}_zs|^2s
=0,
|{\partial}_z s|^2
=|\partial_{\bar{z}}s|^2.
\end{sequation}
From \eqref{eq+2} we find \eqref{lem3-02-1} holds, then $[s]$ is totally real by Lemma \ref{lem3-02}.
From \eqref{eq+1+1} and \eqref{eq+2} we find \eqref{lem3-01-1} holds, then $s$ is minimal by Lemma \ref{lem3-01}.
Since $\varphi$ is isometric,
then
\begin{sequation}\label{eq+4}
ds_M^2=\left(|d^{\mathcal{H}}s({\partial}_z)|^2
+|d^{\mathcal{H}}s({\partial}_{\bar{z}})|^2\right)dzd\bar{z}
=\left(|{\partial}_z s|^2
+|\partial_{\bar{z}}s|^2\right)dzd\bar{z},
\end{sequation}
where in the last equation we use the fact that $s$ is horizontal.
It verifies that $s$ is isometric.
So we get our conclusions.
\end{proof}

Let $\varphi:M\rightarrow \mathbb{H}P^n$
be a linearly full totally real isometric harmonic map.
By \cite {A. Bahy-El-Dien 1991} and \cite{Burstall 1986}, we know that $\varphi$ belongs to the following harmonic sequence in $G(2,2n+2)$,
\begin{sequation}
\cdots
\stackrel{A''_{\varphi_1}}{\longleftarrow} \underline{\varphi}_{-1}
\stackrel{A''_{\varphi_0}}{\longleftarrow} \underline{\varphi}_0=\underline{\varphi}
\stackrel{A'_{\varphi_0}}{\longrightarrow} \underline{\varphi}_{1}
\stackrel{A'_{\varphi_1}}{\longrightarrow} \underline{\varphi}_{2}
\stackrel{A'_{\varphi_2}}{\longrightarrow} \cdots
\stackrel{A'_{\varphi_{m-1}}}{\longrightarrow} \underline{\varphi}_{m}
\stackrel{A'_{\varphi_m}}{\longrightarrow} \underline{\varphi}_{m+1}
\stackrel{A'_{\varphi_{m+1}}}{\longrightarrow} \cdots,
\label{eq:34}
\end{sequation}
where $\underline{\varphi}_{-k}={\texttt{j}}\underline{\varphi}_{k}$
are $2$-dimensional harmonic subbundle of the trivial bundle
$M\times \mathbb{C}^{2n+2}$.
Here $\mathbb{H}P^n$ is seen as the totally geodesic submanifold of
$G(2,2n+2)$.

\begin{prop}\label{prop3-1}
Let $\varphi:M\rightarrow \mathbb{H}P^n$
be a linearly full totally real isometric harmonic map
generating the harmonic sequence \eqref{eq:34}.
If its totally real horizontal minimal lift $[s]$ generates
the following harmonic sequence in $\mathbb{C}P^{2n+1}$,
\begin{sequation}\label{prop3-1-eq1}
\cdots
\stackrel{A''_{f_{-m}}}{\longleftarrow} \underline{f}_{-m}
\stackrel{A''_{f_{-m+1}}}{\longleftarrow} \cdots
\stackrel{A''_{f_{-1}}}{\longleftarrow} \underline{f}_{-1}
\stackrel{A''_{f_0}}{\longleftarrow} \underline{f}_0=[s]
\stackrel{A'_{f_0}}{\longrightarrow} \underline{f}_{1}
\stackrel{A'_{f_1}}{\longrightarrow} \cdots
\stackrel{A'_{f_{m-1}}}{\longrightarrow} \underline{f}_{m}
\stackrel{A'_{f_{m}}}{\longrightarrow} \cdots,
\end{sequation}
where $\underline{f}_{-k},\underline{f}_{k}$
are line bundles,
then for $k=0,1,\cdots,m,\cdots$,
\begin{seqnarray}\label{prop3-1-eq5}
\left\langle f_k,~\texttt{j}f_{-q}\right\rangle
=\left\langle \texttt{j}f_{-k},~f_{q}\right\rangle
=0~(q=k-2,k-1,k),
\end{seqnarray}
and
\begin{sequation}\label{prop3-1-eq6}
\underline{\varphi}_k=\underline{f}_k\oplus \texttt{j}\underline{f}_{-k}.
\end{sequation}
\end{prop}

\begin{proof}
It follows that
$$
\underline{\varphi}_0=\underline{f}_0\oplus \texttt{j}\underline{f}_{0}.
$$
Since
$
\left\langle \partial_z s,~\texttt{j}s\right\rangle
=\left\langle \partial_{\bar{z}} s,~\texttt{j}s\right\rangle
=\left\langle \partial_z s,~\texttt{j}\partial_{\bar{z}}s\right\rangle
=0
$
and
$f_1=\partial_z{s},~f_{-1}=\partial_{\bar{z}}{s},$
then we have
\begin{sequation}\label{prop3-1-eq3}
\left\langle f_1,~\texttt{j}f_0\right\rangle
=\left\langle f_1,~\texttt{j}f_{-1}\right\rangle
=\left\langle \texttt{j}f_{-1},~f_0\right\rangle
=0,
\end{sequation}
which implies
\begin{sequation}\label{prop3-1-eq4}
\underline{\varphi}_1=\underline{f}_1\oplus \texttt{j}\underline{f}_{-1}.
\end{sequation}

By differentiating with respect to $z$ in
$\left\langle f_1,~\texttt{j}f_0\right\rangle=0$,
$\left\langle f_1,~\texttt{j}f_{-1}\right\rangle=0$,
$\left\langle \texttt{j}f_{-1},~f_0\right\rangle=0$
and $\left\langle \texttt{j}f_{-1},~f_1\right\rangle=0$ respectively,
using \eqref{prop3-1-eq3} we obtain
\begin{sequation}\label{prop3-eq+1}
\left\langle f_2,~\texttt{j}f_0\right\rangle
=\left\langle f_2,~\texttt{j}f_{-1}\right\rangle
=\left\langle \texttt{j}f_{-2},~f_0\right\rangle
=\left\langle \texttt{j}f_{-2},~f_1\right\rangle
=0.
\end{sequation}
Then differentiating with respect to $\bar{z}$ in
$\left\langle f_2,~\texttt{j}f_{-1}\right\rangle=0$
we have
\begin{sequation}\label{prop3-eq+2}
\left\langle f_2,~\texttt{j}f_{-2}\right\rangle
=0
\end{sequation}
by \eqref{prop3-1-eq3} and \eqref{prop3-eq+1}.
It follows from \eqref{prop3-eq+1} and \eqref{prop3-eq+2}
that
\begin{sequation}\label{prop3-eq+3}
\underline{\varphi}_2=\underline{f}_2\oplus \texttt{j}\underline{f}_{-2}.
\end{sequation}

In the following, we prove \eqref{prop3-1-eq5} and \eqref{prop3-1-eq6} by induction on $k$.
When $k=1$ and $k=2$ the conclusions hold by \eqref{prop3-1-eq3} -\eqref{prop3-eq+3}.
Suppose the conclusion is true for $k-1$. Consider the case of $k$.
By induction hypotheses we have
\begin{seqnarray}\label{prop3-eq+4}
&\left\langle f_{k-1},~\texttt{j}f_{-(k-3)}\right\rangle
=\left\langle f_{k-1},~\texttt{j}f_{-(k-2)}\right\rangle
=\left\langle f_{k-1},~\texttt{j}f_{-(k-1)}\right\rangle
\nonumber\\
&=\left\langle \texttt{j}f_{-(k-1)},~f_{k-3}\right\rangle
=\left\langle \texttt{j}f_{-(k-1)},~f_{k-2}\right\rangle
=0,
\end{seqnarray}
and
\begin{sequation}\label{prop3-eq+5}
\underline{\varphi}_{k-1}=\underline{f}_{k-1}\oplus \texttt{j}\underline{f}_{-(k-1)}.
\end{sequation}
By differentiating with respect to $z$ in
$\left\langle f_{k-1},~\texttt{j}f_{-(k-2)}\right\rangle=0$,
$\left\langle f_{k-1},~\texttt{j}f_{-(k-1)}\right\rangle=0$,
$\left\langle \texttt{j}f_{-(k-1)},~f_{k-2}\right\rangle=0$
and $\left\langle\texttt{j}f_{-(k-1)},~f_{k-1}\right\rangle=0$
respectively,
using \eqref{prop3-eq+4} we obtain
\begin{sequation}\label{prop3-eq+6}
\left\langle f_k,~\texttt{j}f_{-(k-2)}\right\rangle
=\left\langle f_k,~\texttt{j}f_{-(k-1)}\right\rangle
=\left\langle \texttt{j}f_{-k},~f_{k-2}\right\rangle
=\left\langle \texttt{j}f_{-k},~f_{k-1}\right\rangle
=0.
\end{sequation}
Then differentiating with respect to $\bar{z}$ in
$\left\langle f_k,~\texttt{j}f_{-(k-1)}\right\rangle=0$
we have
\begin{sequation}\label{prop3-eq+7}
\left\langle f_k,~\texttt{j}f_{-k}\right\rangle
=0
\end{sequation}
by \eqref{prop3-eq+4} and \eqref{prop3-eq+6}.
So the case of $k$ holds, which implies \eqref{prop3-1-eq5} holds.

It follows from the properties of harmonic sequence \eqref{eq:34} that
\begin{small}
$$
\underline{\varphi}_k=\text{span}\left\{f_{k}-
\frac{\left\langle f_{k},~\texttt{j}f_{-(k-1)}\right\rangle}{|f_{-(k-1)}|^2}\texttt{j}f_{-(k-1)},
~\texttt{j}f_{-k}-
\frac{\left\langle \texttt{j}f_{-k},~f_{k-1}\right\rangle}{|f_{k-1}|^2} f_{k-1}\right\}.
$$
\end{small}
Then \eqref{prop3-1-eq5} verifies \eqref{prop3-1-eq6}.
So we finish our proof.
\end{proof}

\begin{prop}\label{prop3-2}
Let $\varphi:M\rightarrow \mathbb{H}P^n$
be a linearly full totally real isometric harmonic map
of isotropy order $n$.
Then its totally real horizontal minimal lift
$[s]:M\rightarrow \mathbb{C}P^{2n+1}$ has isotropy order $n$ or $2n+1$.
\end{prop}

\begin{proof}
Since the isotropy order of $\varphi$ is $n$, then the harmonic sequence
\eqref{eq:34} is cyclic, i.e.,
$
\underline{\varphi}_{p}\perp \underline{\varphi}_{q}~(0\leq p<q\leq n),
~\underline{\varphi}_{n+1}=\underline{\varphi}_0.
$
From Proposition \ref{prop3-1},
$
\underline{\varphi}_{k}=\underline{f}_{k}\oplus \texttt{j}\underline{f}_{-k}~(0\leq k\leq n+1).
$
If $\left\langle f_{n+1},~f_{0}\right\rangle\neq0$,
then the isotropy order of $\underline{f}_0=[s]$ is $n$.
If $\left\langle f_{n+1},~f_{0}\right\rangle=0$,
then $\underline{f}_{n+1}=\texttt{j}\underline{f}_{0}$
and $\texttt{j}\underline{f}_{-(n+1)}=\underline{f}_{0}$
by $\underline{\varphi}_{n+1}=\underline{\varphi}_0$.
It follows that the isotropy order of $\underline{f}_0=[s]$ is $2n+1$.
\end{proof}

\section{Totally real flat minimal surfaces}

The following result appears in \cite{Jensen-Liao 1995}, see also
\cite{Kenmotsu 1985} and \cite{Bolton-Woodward 1992}.
\begin{theorem}
If $f:\mathbb{C}\rightarrow \mathbb{C}P^n$ is a linearly full totally real harmonic map with induced
metric $f^*ds^2=2dzd\overline{z}$. Then up to a unitary equivalence,
$f=\left[V_0^{(n)}\right]$, where
\begin{sequation}
V_0^{(n)}(z)=\begin{pmatrix}
e^{a_0z-\overline{a}_0\overline{z}}\xi^0,
&e^{a_1z-\overline{a}_1\overline{z}}\xi^1,
&\cdots,
&e^{a_nz-\overline{a}_n\overline{z}}\xi^n
\end{pmatrix}^T,
\label{eq3-13}
\end{sequation}
and
$a_k=e^{\texttt{i}\theta_k},~\xi^k=\sqrt{r_k}~(r_k>0)$
for $k=0,1,\cdots,n$, satisfying
$0=\theta_0<\theta_1<\cdots <\theta_n<2\pi,
~r_0+r_1+\cdots+r_n=1$.
\label{thm3-2}
\end{theorem}

In the following we will study linearly full totally real flat minimal surfaces in $\mathbb{H}P^n$.
Let $\varphi:\mathbb{C}\rightarrow \mathbb{H}P^n$ be a linearly full totally real flat minimal surface,
from Theorem \ref{thm3-1}, we know there exists a totally real flat minimal surface $[s]:\mathbb{C}\rightarrow\mathbb{C}P^{2n+1}$ as
the horizontal lift of $\varphi$.
Here $[s]$ may not be linearly full in $\mathbb{C}P^{2n+1}$, then it lies in $\mathbb{C}P^{m}\subset\mathbb{C}P^{2n+1}$ for
$n\leq m\leq 2n+1$, where $m\geq n$ by $\varphi$ being linearly full in $\mathbb{H}P^n$. Using Theorem \ref{thm3-2}, we know $s$ is given by \eqref{eq3-13}, up to $U(2n+2)$. But the isometry group of $\mathbb{H}P^n$ is $Sp(n+1)=\left\{U\in U(2n+2),~UJ_{n+1}U^{T}=J_{n+1}\right\}$, a subgroup of $U(2n+2)$, then in order to give the explicit expression
of $\varphi$, we need to characterize $s$, up to $Sp(n+1)$.

\begin{theorem}
Let $\varphi:\mathbb{C}\rightarrow \mathbb{H}P^n$ be a linearly full totally real flat minimal surface
of isotropy order $n$
with the induced metric $\varphi^*ds^2=2dzd\overline{z}$,
then up to a symplectic isometry of $\mathbb{H}P^n$,
$\varphi$ lies in $\mathbb{C}P^n~(\subset \mathbb{H}P^n)$ given by
\begin{sequation}\label{sec4+00}
\varphi=\begin{bmatrix}
e^{a_0z-\overline{a}_0\overline{z}},
&e^{a_1z-\overline{a}_1\overline{z}},
&\cdots,
&e^{a_{n}z-\overline{a}_{n}\overline{z}}
\end{bmatrix}^T,
\end{sequation}
where $a_k=e^{\texttt{i}\frac{2k\pi}{n+1}}~(k=0,1,\cdots,n)$
(the Clifford solution in $\mathbb{C}P^n$), or $a_k=e^{\texttt{i}\frac{k\pi}{n+1}}~(k=0,1,\cdots,n)$.
\label{thm3}
\end{theorem}

\begin{proof}
Let $\varphi:\mathbb{C}\rightarrow \mathbb{H}P^n$ be a linearly full totally real flat minimal surface
of isotropy order $n$
with the induced metric $\varphi^*ds^2_{\mathbb{H}P^n}=2dzd\overline{z}$.
From Theorem \ref{thm3-1} and Proposition \ref{prop3-2}, there exists a totally real horizontal minimal lift
$\underline{f}_0=[s]:\mathbb{C}\rightarrow \mathbb{C}P^{2n+1}$ of isotropy order $n$ or $2n+1$ with the induced metric $f_0^*ds^2_{\mathbb{C}P^{2n+1}}=2dzd\overline{z}$, such that
\begin{sequation}\label{eq000}
\underline{\varphi}=\underline{f}_0\oplus \texttt{j}\underline{f}_0.
\end{sequation}
Then by Theorem \ref{thm3-2}, there exists a unitary matrix $U\in U(2n+2)$ such that
$s=U{V}_0^{(m)}~(n\leq m\leq 2n+1)$
satisfying
$\left\langle \partial^k_z s,~s\right\rangle
=\left\langle \partial^k_z s,~\texttt{j}s\right\rangle
=\left\langle \partial^k_{\bar{z}} s,~\texttt{j}s\right\rangle
=0~(k=1,\cdots,n)$.
Here, $V_0^{(m)}$ is a column vector of order $2n+2$ as follows,
\begin{small}
$$V_0^{(m)}(z)=\begin{pmatrix}
e^{a_0z-\overline{a}_0\overline{z}}\xi^0,
&e^{a_1z-\overline{a}_1\overline{z}}\xi^1,
&\cdots,
&e^{a_mz-\overline{a}_m\overline{z}}\xi^m,
&0,
&\cdots,
&0
\end{pmatrix}^T.
$$
\end{small}
Since $\texttt{j}s=J_{n+1}\bar{s}$,
setting $W=U^TJ_{n+1}U=(w_{ij})$, we know $W$ is an anti-symmetric unitary matrix of order $2n+2$.
Then for any $k=1,\cdots,n$,
\begin{small}
\begin{eqnarray}
\sum_{j=0}^m a_j^kr_j=
\sum_{i,j=0}^m \left(w_{ij}\xi^i\xi^ja_j^k
e^{(a_i+a_j)z-(\overline{a}_i+\overline{a}_j)\overline{z}}\right)
=\sum_{i,j=0}^m \left(w_{ij}\xi^i\xi^j(-\bar{a}_j)^k
e^{(a_i+a_j)z-(\overline{a}_i+\overline{a}_j)\overline{z}}\right)
=0.
\label{thm3-6}
\end{eqnarray}
\end{small}

For a given $W$,
define a set
\begin{small}
$$G_W\triangleq \left\{U\in U(2n+2)| \
U^TJ_{n+1}U=W\right\}.
$$
\end{small}
The following can be easily
checked,\\
(i) $\forall ~A\in Sp(n+1), ~U\in G_W$, we have that $AU\in G_W$;
\\
(ii) $\forall ~U, V \in G_W, ~\exists ~A=UV^* \in Sp(n+1) ~s.t. ~U=AV$.

In the following we discuss $W$ in two cases of the isotropy order of $\underline{f}_0$
being $n$ or $2n+1$.

If the isotropy order of $\underline{f}_0$ is $n$, then by Proposition \ref{prop3-2},
\begin{equation}\label{sec4+1}
f_{n+1}=af_0+b\texttt{j}f_0,
~\texttt{j}f_{-(n+1)}=\lambda\left(-\bar{b}f_0+\bar{a}\texttt{j}f_0\right),
\end{equation}
where $\lambda$ is a unit complex and $a(\neq 0),b$ are complexes satisfying $|a|^2+|b|^2=1$.
It follows from \eqref{sec4+1} that
$f_{n+1}=\lambda f_{-(n+1)}$,
which implies by $f_{n+1}=\partial_z^{n+1}s$ and $f_{-(n+1)}=\partial_{\bar{z}}^{n+1}s$,
\begin{sequation}\label{sec4+3}
a_j^{n+1}e^{a_jz-\bar{a}_j\bar{z}}\xi^j
=\lambda(-\bar{a}_j)^{n+1}e^{a_jz-\bar{a}_j\bar{z}}\xi^j
~(j=0,1,\cdots,m).
\end{sequation}
Since $a_0=1$ and $a_i\neq a_j~(i\neq j)$, then from \eqref{sec4+3},
$a_j^{2n+2}=1~(j=0,1,\cdots,m)$.

In this case, if $n\leq m<2n+1$ or $m=2n+1$($n$ is odd), then we claim that $W$ belongs to the following type
\begin{sequation}\label{w+1}
W=\begin{pmatrix}
0_{(m+1)\times(m+1)} & *\\
* & *
\end{pmatrix}.
\end{sequation}
In fact if $\mathbf{\forall ~i,j=0,1,\cdots,m,~a_i+a_j\neq 0}$,
then we have $w_{ij}=0$ by \eqref{thm3-6}, which implies that $W$
is just the type of \eqref{w+1}.
If $\mathbf{\forall ~i,~\exists ~j~s.t.~a_i+a_j=0}$, then $m$ must be odd.
Set $m=2p+1$. Then, for any $i=0,1,\cdots,p$, without of generality, we may assume that
$a_i+a_{p+1+i}=0$.
Applying it in \eqref{thm3-6}, we obtain that for any $i=0,1,\cdots,m$,
$w_{ij}=0~(j\neq p+1+i~(\text{mod} ~m+1))$,
and
$\sum_{i=0}^pw_{i,p+1+i}\xi^i\xi^{p+1+i}a_i^{2\tau+1}
=\sum_{i=0}^pw_{i,p+1+i}\xi^i\xi^{p+1+i}\bar{a}_i^{2\tau+1}
=0
~(\tau=0,1,\cdots,q)$,
where $2q+1$ is the maximum odd number not greater than $n$.
It follows that
\begin{small}
$$
\begin{pmatrix}
a_0 & a_1 & \cdots & a_p\\
a_0^3 & a_1^3 & \cdots & a_p^3\\
\vdots & \vdots & \ddots & \vdots\\
a_0^{2q+1} & a_1^{2q+1} & \cdots & a_p^{2q+1}\\
\bar{a}_0 & \bar{a}_1 & \cdots & \bar{a}_p\\
\bar{a}_0^3 & \bar{a}_1^3 & \cdots & \bar{a}_p^3\\
\vdots & \vdots & \ddots & \vdots\\
\bar{a}_0^{2q+1} & \bar{a}_1^{2q+1} & \cdots & \bar{a}_p^{2q+1}\\
\end{pmatrix}
\begin{pmatrix}
w_{0,p+1}\xi^0\xi^{p+1}\\
w_{1,p+2}\xi^1\xi^{p+2}\\
\vdots\\
w_{p,2p+1}\xi^p\xi^{2p+1}
\end{pmatrix}
=\begin{pmatrix}
0\\
0\\
\vdots\\
0
\end{pmatrix}.
$$
\end{small}
Since $a_j^{2n+2}=1$, then $\bar{a}_j=a_j^{2n+1}$. It means that the coefficient matrix of the above equation is a Vandermonde matrix of $(2q+2)\times(p+1)$, whose rank is $p+1$.
Here $2q+2\geq p+1$ by $n\leq m<2n+1$ or $m=2n+1$($n$ is odd).
So, we have $w_{i,p+1+i}=0~(i=0,1,\cdots,p)$.
Thus $W$ is also just the type of \eqref{w+1}.
If $\mathbf{\exists ~i~s.t.\forall ~j,~a_i+a_j\neq 0
~and ~\exists~k,l~s.t.~a_k+a_l=0}$,
by the similar discussion as the above, we get $W$ is also just the type of \eqref{w+1}.

If $n<m\leq 2n+1$ then $W$ in \eqref{w+1} is a degenerate matrix, hence there doesn't exist this case. Thus $m=n$ and
\begin{sequation}\label{thm3-6+1}
W=\begin{pmatrix}
0_{(n+1)\times(n+1)} & W_{12}\\
* & *
\end{pmatrix}
~\text{with}
~W_{12}=\begin{pmatrix}
w_{0,n+1} & \cdots & w_{0,2n+1}\\
\vdots & \ddots & \vdots\\
w_{n,n+1} & \cdots & w_{n,2n+1}
\end{pmatrix}.
\end{sequation}
From \eqref{thm3-6+1}, the corresponding matrix $U\in G_W$ can be expressed as follow:
\begin{small}
$$
U=\begin{pmatrix}
1 & 0 & \cdots & 0 & 0 & 0 & \cdots & 0\\
0 & 0 & \cdots & 0 & -w_{0,n+1} & -w_{0,n+2} & \cdots & -w_{0,2n+1}\\
\vdots & \vdots & \ddots & \vdots & \vdots & \vdots & \ddots & \vdots\\
0 & 0 & \cdots & 1 & 0 & 0 & \cdots & 0\\
0 & 0 & \cdots & 0 & -w_{n,n+1} & -w_{n,n+2} & \cdots & -w_{n,2n+1}\\
\end{pmatrix}.
$$
\end{small}
Then we get the horizontal lift of $\varphi$
\begin{sequation}\label{w+3}
s(z)=UV_0^{(n)}(z)=\begin{pmatrix}
e^{a_0z-\overline{a}_0\overline{z}}\xi^0,
&0,
&e^{a_1z-\overline{a}_1\overline{z}}\xi^1,
&0,
&\cdots,
&e^{a_nz-\overline{a}_n\overline{z}}\xi^n,
&0
\end{pmatrix}^T,
\end{sequation}
which implies
$$
\varphi=\begin{bmatrix}
e^{a_0z-\overline{a}_0\overline{z}}\xi^0,
&e^{a_1z-\overline{a}_1\overline{z}}\xi^1,
&\cdots,
&e^{a_nz-\overline{a}_n\overline{z}}\xi^n
\end{bmatrix}^T.
$$
So, the image of $\varphi$ lies in $\mathbb{C}P^n$.
Since the isotropy order of $\varphi$ is $n$, then it follows from (\cite{Jensen-Liao 1995},~Proposition 4.1) that
$\varphi$ is just the Clifford solution of \eqref{sec4+00} in $\mathbb{C}P^n$, up to $U(n+1)$ ($\subset Sp(n+1)$).

Now we discuss the case of $m=2n+1$($n$ is even). Noticing $2q+2<n+1$ in this case, we can't follow the above arguments. Since $a_j^{2n+2}=1$ and $a_j=e^{\texttt{i}\theta_j}$ with $0=\theta_0<\theta_1<\cdots <\theta_{2n+1}<2\pi$, then $a_j=e^{\texttt{i}\frac{j\pi}{n+1}}~(j=0,1,\cdots,2n+1)$,
which implies $a_{\tau}+a_{n+1+\tau}=0~(\tau=0,1,\cdots,n)$.
Using $f_{n+1}=a f_0+b\texttt{j}f_0$ again, we obtain
$
(a_i^{n+1}-a)e^{a_iz-\bar{a}_i\bar{z}}\xi^i
=b\sum_{j=0}^{2n+1}\bar{w}_{ij}e^{\bar{a}_j\bar{z}-a_jz}\xi^j
~(i=0,1,\cdots,2n+1)
$, which implies that for any $i=0,1,\cdots,2n+1$,
$w_{ij}=0~(j\neq n+1+i~(\text{mod} ~2n+2))$, and
\begin{sequation}\label{w+2}
(a_\tau^{n+1}-a)\xi^\tau=b\bar{w}_{\tau,n+1+\tau}\xi^{n+1+\tau},
~(a_{\tau}^{n+1}+a)\xi^{n+1+\tau}
=b\bar{w}_{\tau,n+1+\tau}\xi^{\tau}~(\tau=0,1,\cdots,n).
\end{sequation}
Set $n=2p$, from \eqref{w+2} we get $a=\frac{r_0-r_{n+1}}{r_0+r_{n+1}}
=(-1)^{\tau}\frac{r_\tau-r_{n+1+\tau}}{r_\tau+r_{n+1+\tau}}
~(\tau=1,\cdots,n)$,
which implies
$\frac{r_{n+1}}{r_0}=\frac{r_{n+1+2\zeta}}{r_{2\zeta}}
=\frac{r_{2\zeta-1}}{r_{n+2\zeta}}~(\zeta=1,\cdots,p)$, where $r_0\neq r_{n+1}$ by $a\neq 0$.
Substituting it into \eqref{w+2} and the first equation of \eqref{thm3-6} respectively, a straightforward calculation shows ${w}_{\tau,n+1+\tau}=(-1)^\tau w_{0,n+1}~(\tau=1,\cdots,n)$
and
$
r_0+\sum_{\tau=1}^{n}\left(a_{2\tau}^kr_{2\tau}\right)=0~(k=1,\cdots,n),
$
which implies
by the fact of $\sum_{\tau=1}^n a_{2\tau}^k=-1$,
\begin{small}
$$
\begin{pmatrix}
a_2 & a_{4} & \cdots & a_{2n}\\
a_2^2 & a_{4}^2 & \cdots & a_{2n}^2\\
\vdots & \vdots & \ddots & \vdots\\
a_2^{n} & a_{4}^{n} & \cdots & a_{2n}^{n}
\end{pmatrix}
\begin{pmatrix}
r_2\\
r_4\\
\vdots\\
r_{2n}
\end{pmatrix}
=\begin{pmatrix}
a_2 & a_{4} & \cdots & a_{2n}\\
a_2^2 & a_{4}^2 & \cdots & a_{2n}^2\\
\vdots & \vdots & \ddots & \vdots\\
a_2^{n} & a_{4}^{n} & \cdots & a_{2n}^{n}
\end{pmatrix}
\begin{pmatrix}
r_0\\
r_0\\
\vdots\\
r_0
\end{pmatrix}.
$$
\end{small}
Since the determinant of coefficient matrix of the above equation is
$a_2a_4\cdots a_{2n}\prod_{1\leq j<i\leq n}(a_{2i}-a_{2j})\neq 0$,
we have $r_{2\tau}=r_0$,
which implies $r_{2\tau+1}=r_1~(\tau=1,\cdots,n)$.
It follows that
\begin{small}
\begin{equation*}\label{thm3-6+1+1}
W=\begin{pmatrix}
0_{(n+1)\times(n+1)} & W_{12}\\
* & *
\end{pmatrix}
~\text{with}
~W_{12}=\begin{pmatrix}
(-1)^0w_{0,n+1} & \cdots & 0\\
\vdots & \ddots & \vdots\\
0 & \cdots & (-1)^nw_{0,n+1}
\end{pmatrix},
\end{equation*}
\end{small}
where $|w_{0,n+1}|^2=1$. By the similar calculation we can get the horizontal lift of $\varphi$
\begin{tiny}
\begin{seqnarray}\label{w+4}
s(z)=\left(
e^{a_0z-\overline{a}_0\overline{z}}\xi^0,
(-1)^1w_{0,n+1}e^{a_{n+1}z-\overline{a}_{n+1}\overline{z}}
\xi^{n+1},
\cdots,
e^{a_nz-\overline{a}_n\overline{z}}\xi^n,
(-1)^{n+1}w_{0,n+1}e^{a_{2n+1}z-\overline{a}_{2n+1}\overline{z}}
\xi^{2n+1}
\right)^T,
\end{seqnarray}
\end{tiny}
which implies by $a_\tau+a_{n+1+\tau}=0,~\xi^{2\tau}=\xi^0,~\xi^{2\tau+1}=\xi^1
~(\tau=0,1,\cdots,n)$,
\begin{tiny}
\begin{equation*}\label{sec4+0}
\varphi=\left[
(\xi^0+(-1)w_{0,n+1}\xi^1\texttt{j})e^{a_0z-\overline{a}_0\overline{z}},
(\xi^1+(-1)^2w_{0,n+1}\xi^0\texttt{j})e^{a_1z-\overline{a}_1\overline{z}},
\cdots,
(\xi^0+(-1)^{n+1}w_{0,n+1}\xi^1\texttt{j})e^{a_{n}z-\overline{a}_{n}\overline{z}}
\right]^T.
\end{equation*}
\end{tiny}
After removing the projective factor $(\xi^0-w_{0,n+1}\xi^1\texttt{j})$, $\varphi$ is the surface \eqref{sec4+00} with $a_k=e^{\texttt{i}\frac{k\pi}{n+1}}$ and
$n$ being even, up to $Sp(n+1)$.

If the isotropy order of $\underline{f}_0$ is $2n+1$, then $\underline{f}_0$ is
the Clifford solution (up to congruence) in $\mathbb{C}P^{2n+1}$ by (\cite{Jensen-Liao 1995},~Proposition 4.1),
that is, $s=U{V}_0^{(2n+1)}$ with
\begin{small}
$$V_0^{(2n+1)}(z)=\begin{pmatrix}
e^{a_0z-\overline{a}_0\overline{z}}\xi^0,
&e^{a_1z-\overline{a}_1\overline{z}}\xi^1,
&\cdots,
&e^{a_{2n+1}z-\overline{a}_{2n+1}\overline{z}}\xi^{2n+1}
\end{pmatrix}^T,
$$
\end{small}
where $a_k=e^{\texttt{i}\frac{k\pi}{n+1}},~\xi^k=\sqrt{\frac{1}{2n+2}}~(k=0,1,\cdots,2n+1)$.
Since $\underline{f}_{n+1}=\texttt{j}\underline{f}_{0}$, then
\begin{small}
\begin{equation*}\label{thm3-6+1+1}
W=\begin{pmatrix}
0_{(n+1)\times(n+1)} & W_{12}\\
* & *
\end{pmatrix}
~\text{with}
~W_{12}=\begin{pmatrix}
e^{\texttt{i}\pi n} & \cdots & 0\\
\vdots & \ddots & \vdots\\
0 & \cdots & e^{\texttt{i}\pi 0}
\end{pmatrix}.
\end{equation*}
\end{small}
By the similar computation we can get the horizontal lift of $\varphi$
\begin{small}
\begin{eqnarray}\label{w+5}
s(z)=\left(
e^{a_0z-\overline{a}_0\overline{z}}\xi^0,
(-1)^{n+1}e^{a_{n+1}z-\overline{a}_{n+1}\overline{z}}
\xi^{n+1},
\cdots,
e^{a_nz-\overline{a}_n\overline{z}}\xi^n,
(-1)e^{a_{2n+1}z-\overline{a}_{2n+1}\overline{z}}
\xi^{2n+1}
\right)^T,
\end{eqnarray}
\end{small}
which implies by $a_k=e^{\texttt{i}\frac{k\pi}{n+1}}$
and $\xi^k=\sqrt{\frac{1}{2n+2}}$,
\begin{small}
\begin{equation*}\label{sec4+0}
\varphi=\begin{bmatrix}
(1+(-1)^{n+1}\texttt{j})e^{a_0z-\overline{a}_0\overline{z}}\xi^0,
&(1+(-1)^{n}\texttt{j})e^{a_1z-\overline{a}_1\overline{z}}\xi^1,
&\cdots,
&(1+(-1)^{1}\texttt{j})e^{a_{n}z-\overline{a}_{n}\overline{z}}\xi^{n}
\end{bmatrix}^T.
\end{equation*}
\end{small}
After removing the projective factors $\xi^k$ and $(1+\texttt{j})$ or $(1-\texttt{j})$, $\varphi$ is the surface \eqref{sec4+00} with $a_k=e^{\texttt{i}\frac{k\pi}{n+1}}$, up to $Sp(n+1)$.
In summary, we finish our proofs.
\end{proof}
\begin{remark}
Two surfaces of isotropy order $n$ in $\mathbb{H}P^n$ given by Theorem \ref{thm3} have different horizontal lifts.
For the surface \eqref{sec4+00} with $a_k=e^{\texttt{i}\frac{2k\pi}{n+1}}$, the horizontal lift
\eqref{w+3} is not linearly full in $\mathbb{C}P^{2n+1}$, whose isotropy order is $n$.
However for the surface \eqref{sec4+00} with $a_k=e^{\texttt{i}\frac{k\pi}{n+1}}$, the horizontal lifts
\eqref{w+4}($n$ is even) and \eqref{w+5} are both linearly full in $\mathbb{C}P^{2n+1}$, whose isotropy order is $n$ and $2n+1$ respectively.
\end{remark}

{\bf{Acknowledgments}}~
The authors would like to appreciate the referee for carefully reading the manuscript and pointing out a mistake in the proof
of Proposition 3.5.
The authors were supported by NSF in China (No. 11501548, 11501505).


\begin{thebibliography}{90}


\bibitem{A. Bahy-El-Dien 1991}
A.Bahy-El-Dien and J.C.Wood, \emph{The explicit construction of
all harmonic two-spheres in quaternionic projective spaces}, Proc. London Math. Soc., 62(1991), 202-224.
\bibitem{Bando-Ohnita}
S.Bando and Y.Ohnita, \emph{Minimal $2$-spheres with constant curvature in $\mathbb{C}P^n$}, J. Math. Soc. Japan, 39(1987), 477-487.
\bibitem{Besse 1987}
A.L.Besse, \emph{Einstein Manifolds}, Springer-Verlag Berlin Heidelberg, 1987.
\bibitem{J. Bolton 1988}
J.Bolton, G.R.Jensen, M.Rigoli and L.M.Woodward, \emph{On
conformal minimal immersions of $S^2$ into $\mathbb{C}P^N$}, Math. Ann.,
279(1988), 599-620.
\bibitem{Bolton-Woodward 1992}
J.Bolton and L.M.Woodward, \emph{Minimal surfaces in $\mathbb{C}P^n$
with constant curvature and K\"ahler angle}, Proc. Cambridge Philos. Soc.,
112(1992), 287-296.
\bibitem{Burstall 1986}
F.E.Burstall and J.C.Wood, \emph{The construction of harmonic
maps into complex Grassmannians}, J. Diff. Geom., 23(1986), 255-297.
\bibitem{FeiHe}
J.Fei and L.He, \emph{Classification of homogeneous minimal
immersions from $S^2$ to $\mathbb{H}P^n$},
Ann. Mat. Pur. Appl., 196(2017), 2213-2237.
\bibitem{S.Funabashi 1978}
S.Funabashi, \emph{Totally real submanifolds of a quaternionic
Kaehlerian manifold}, Kodai Math. Sem. Rep., 29(1978), 261-270.
\bibitem{S.Funabashi 1979}
S.Funabashi, \emph{Totally complex submanifolds of a quaternionic
Kaehlerian manifold}, Kodai Math. J., 2(1979), 314-336.
\bibitem{HeJiao 2014}
L.He and X.X.Jiao, \emph{Classification of conformal minimal immersions of constant
curvature from $S^2$ to $\mathbb{H}P^2$}, Math. Ann., 359(2014), 663-694.
\bibitem{HeJiao 2015}
L.He and X.X.Jiao, \emph{On conformal minimal immersions of constant
curvature from $S^2$ to $\mathbb{H}P^n$}, Math. Z., 280(2015), 851-871.
\bibitem{HeJiao 2015-2}
L.He and X.X.Jiao, \emph{On conformal minimal immersions of $S^2$ in $\mathbb{H}P^n$
with parallel second fundamental form}, Ann. Mat. Pur. Appl., 194(2015), 1301-1317.
\bibitem{HeWang2005}
Y.J.He and C.P.Wang, \emph{Totally real minimal 2-spheres in
quaternionic projective space}, Sci in China Ser A(Math),
48(3)(2005), 341-349.
\bibitem{LiZhang 2005}
J.Y.Li and X.Zhang, \emph{Quaternionic maps between quaternionic K\"ahler manifolds}, Math. Z., 250(2005), 523-537.
\bibitem{Jensen-Liao 1995}
G.R.Jensen and R.J.Liao, \emph{Families of flat minimal tori in $\mathbb{C}P^n$},
J. Diff. Geom., 42(1)(1995), 113-132.
\bibitem{Kenmotsu 1985}
K.Kenmotsu, \emph{On minimal immersions of $R^2$ into $\mathbb{C}P^n$},
J. Math. Soc. Japan, 37(1985), 665-682.
\bibitem{Ma 2005}
H.Ma and Y.J.Ma, \emph{Totally real minimal tori in $\mathbb{C}P^2$},
Math. Z., 249(2)(2005), 241-267.
\bibitem{Salamon 1982}
S.Salamon, \emph{Quaternionic Kaehler manifolds}, Invent. Math., 67(1982), 143-171.
\bibitem{K.Tsukada 1985}
K.Tsukada, \emph{Parallel submanifolds in a quaternion projective space}, Osaka J. Math., 22(1985), 187-241.
\bibitem{K. Uhlenbeck 1989}
K.Uhlenbeck, \emph{Harmonic maps into lie groups}, J. Diff. Geom., 30(1989), 1-50.
\bibitem{S.Udagawa 1997}
S.Udagawa, \emph{Harmonic tori in quaternionic projective 3-spaces},
Proc. Amer. Math. Soc., 125(1)(1997), 275-285.
\end{thebibliography}
\end{document}